\documentclass[a4paper,12pt]{amsart}

\usepackage{amssymb}
\usepackage{latexsym}
\usepackage[all]{xy}
\usepackage{amsmath}
\usepackage{amscd}
\usepackage{graphicx}
\usepackage{palatino}

\title{Lagrangian mapping class groups from a group homological 
point of view}

\author{Takuya Sakasai}
\address{Department of Mathematics, 
Tokyo Institute of Technology,
2-12-1 Oh-okayama, Meguro-ku, 
Tokyo, 152-8551, Japan}
\email{sakasai@math.titech.ac.jp}
\date{\today}


\newtheorem{thm}{Theorem}[section]
\newtheorem{prop}[thm]{Proposition}
\newtheorem{lem}[thm]{Lemma}
\newtheorem{cor}[thm]{Corollary}

\theoremstyle{definition}

\newtheorem{remark}[thm]{Remark}

\numberwithin{equation}{section}

\begin{document}

\newcommand{\Mg}{\mathcal{M}_g}
\newcommand{\Mgb}{\mathcal{M}_{g,1}}
\newcommand{\Ig}{\mathcal{I}_g}
\newcommand{\Igb}{\mathcal{I}_{g,1}}
\newcommand{\Sg}{\Sigma_g}
\newcommand{\Sgb}{\Sigma_{g,1}}
\newcommand{\Symp}[1]{Sp(2g,\mathbb{#1})}
\newcommand{\Lg}{\mathcal{L}_g}
\newcommand{\Lgb}{\mathcal{L}_{g,1}}
\newcommand{\ILg}{\mathcal{IL}_g}
\newcommand{\ILgb}{\mathcal{IL}_{g,1}}

\newcommand{\ur}[1]{urSp(#1)}
\newcommand{\urp}[1]{urSp^+(#1)}
\newcommand{\Hq}{H_{\mathbb{Q}}}

\newcommand{\Ker}{\mathop{\mathrm{Ker}}\nolimits}
\newcommand{\Hom}{\mathop{\mathrm{Hom}}\nolimits}

\newcommand{\Out}{\mathop{\mathrm{Out}}\nolimits}
\newcommand{\Aut}{\mathop{\mathrm{Aut}}\nolimits}
\renewcommand{\Im}{\mathop{\mathrm{Im}}\nolimits}
\newcommand{\Q}{\mathbb{Q}}
\newcommand{\Z}{\mathbb{Z}}

\begin{abstract}
We focus on two kinds of infinite index subgroups of the mapping class
group of a surface associated with a Lagrangian submodule of the first
homology of a surface. These subgroups, called Lagrangian mapping class 
groups, are known to play important roles in the interaction between
the mapping class group and finite-type invariants of 
3-manifolds. 
In this paper, we discuss these groups from a group (co)homological point of 
view. The results include 
the determination of their abelianizations, lower bounds of 
the second homology and remarks on the (co)homology of higher degrees. 
As a by-product of this investigation, we determine the second homology 
of the mapping class group of a surface of genus $3$.
\end{abstract}

\renewcommand\baselinestretch{1.1}
\setlength{\baselineskip}{16pt}

\newcounter{fig}
\setcounter{fig}{0}

\maketitle

\section{Introduction}\label{sec:intro}
Let $\Sg$ be a closed oriented connected surface of genus $g$ and 
let $H_g$ be an oriented handlebody  of the same genus. 
As depicted in Figure \ref{fig:surface}, we put $H_g$ 
in the standard position in $\mathbb{R}^3$ and consider 
$\Sg$ to be the boundary of  $H_g$. 
We fix a basis $\{x_1, x_2,\ldots,x_g, y_1,y_2,\ldots,y_g\}$ 
of $H:=H_1 (\Sg)$ as in the figure 
so that $\Ker (H_1 (\Sg) \to H_1 (H_g))$ coincides with 
the submodule $L$ of $H$ generated by $\{x_1, x_2,\ldots,x_g\}$. 

\begin{figure}[htbp]
\begin{center}
\includegraphics[width=0.55\textwidth]{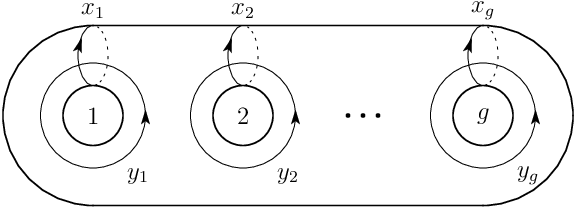}
\end{center}
\caption{A symplectic basis of $H_1 (\Sg)$}
\label{fig:surface}
\end{figure}

The module $H$ has a natural non-degenerate anti-symmetric 
bilinear form $\mu:H \otimes H \to \Z$ called 
the intersection pairing. 
It is easy to see that $L$ is a maximal direct summand of $H$ on 
which $\mu$ 
restricts to $0$. 
Such a submodule is said to be {\it Lagrangian}. 
By using the pairing $\mu$, we can naturally identify 
the quotient module $H/L$, the dual module $L^\ast:=\Hom(L,\Z)$ and 
the submodule $L_y$ of $H$ generated by $\{y_1, y_2,\ldots,y_g\}$. 

The {\it mapping class group} $\Mg$ of $\Sg$ is the group of 
isotopy classes of orientation preserving 
self-diffeomorphisms of $\Sg$. 
In this paper, we focus on subgroups of $\Mg$ 
associated with the above fixed Lagrangian submodule $L$ of $H$. 
More precisely, two subgroups 
\begin{align*}
\Lg \ &:=\{f \in \Mg \mid f_\ast (L)=L\}, \\
\ILg&:=\{f \in \Mg \mid f_\ast |_L = \mathrm{id}_L\}
\end{align*}
are studied through their group (co)homology, where 
$f_\ast$ denotes the induced automorphism of $H$ 
for $f \in \Mg$. 
We have $\ILg \subset \Lg$ by definition and call them 
{\it Lagrangian mapping class groups} or 
{\it Lagrangian subgroups}. 
The {\it Torelli group} $\Ig$ is defined by
\[\Ig := \{f \in \Mg \mid f_\ast = \mathrm{id}_H\}.\]

One motivation by which the author started to study the groups 
$\Lg$ and $\ILg$ is 
the fact that they are {\it infinite\/} index subgroups of $\Mg$ 
{\it including\/} $\Ig$. 
The importance to study such a kind of subgroups 
will be explained in Section \ref{subsec:oddMMM} 
with the relationship to the (non-)triviality problem of 
{\it even Miller-Morita-Mumford classes} $e_{2i} \in H^{4i}(\Mg;\Q)$ 
pulled back to $H^{4i}(\Ig;\Q)$. 

The study of Lagrangian subgroups has been done by several researchers and 
here we recall them briefly. 
Hirose studied a generating system of $\Lg$ in \cite{hirose}, 
where $\Lg$ is called the {\it homological handlebody group}. 
In fact, the group $\Lg$ can be seen as a homological extension of the 
handlebody mapping class group $\mathcal{H}_g$. 
Recall that the group $\mathcal{H}_g$ is the subgroup of $\Mg$ consisting of 
isotopy classes of orientation preserving 
self-diffeomorphisms of $\Sg=\partial H_g$ that can be extended to 
self-diffeomorphisms of the handlebody $H_g$. 
We can easily check that $\Lg = \mathcal{H}_g \Ig$. 
Note that, prior to Hirose's work, Birman gave a generating 
set of $\Lg/\Ig \cong \mathcal{H}_g/(\mathcal{H}_g \cap \Ig)$ in \cite{bi2} and we can give 
a generating set of $\Lg$ by combining her result with Johnson's finite generating 
set of $\Ig$ (see \cite{jo1}). 

As for $\ILg$, Levine gave a series of investigations 
in \cite{le1, le3, le4}. 
He defined a filtration of $\ILg$ called the {\it Lagrangian filtration}, 
which is analogous to the Johnson filtration of $\Ig$, by modifying 
the theory of Johnson homomorphisms 
so that it conforms well to $\ILg$. Then he gave an application of this 
filtration to the theory of homology 3-spheres. 

Recently, the groups $\Lg$ and $\ILg$ appear and play important roles in 
the theory of finite-type invariants of 3-manifolds. See 
Andersen-Bene-Meilhan-Penner \cite{abmp}, 
Cheptea-Habiro-Massuyeau \cite{chm}, 
Cheptea-Le \cite{cl} (with a slightly different definition) and \\
Garoufalidis-Levine \cite{gl2} 
for example. However, it seems that the groups $\Lg$ and $\ILg$ 
have been studied separately. In this paper, we 
put $\Lg$ on the top of the Lagrangian filtration 
of $\ILg$ and study them simultaneously as in the case of $\Mg$ and $\Ig$. 

Here we mention the contents of this paper. 
We first summarize the notation and fundamental facts on $\Lg$ and 
$\ILg$ in Section \ref{sec:LagMCG}. Then we will discuss 
the following in order. 
\begin{itemize}
\item Section \ref{sec:H1ILgb}: Computation of $H_1 (\ILg)$ 

\item Section \ref{sec:H_ur}: 
Computations of $H_1 (\Lg/\Ig)$ and $H_2(\Lg/\Ig)$ 

\item Section \ref{sec:H1Lgb}: 
Computation of $H_1(\Lg)$ and a lower bound of $H_2(\Lg)$ 

\item Section \ref{sec:remark}: Remarks on higher (co)homology 
of $\Lg$ and $\ILg$
\end{itemize}
Precisely speaking, we study in Sections \ref{sec:H1ILgb} and 
\ref{sec:H1Lgb} the Lagrangian mapping class groups of a surface 
with one boundary component 
and then derive the statements for those of a closed surface 
in Section \ref{sec:closed}. 

As a by-product, we will give a remark that the second homology of 
the full mapping class group of genus $3$ has $\Z_2$ as a direct summand 
(Theorem \ref{thm:H2M3b} and Corollary \ref{cor:H2M3}), 
where this homology group has been 
almost determined by Korkmaz-Stipsicz \cite{ks} 
up to this $\Z_2$ summand. 

In this paper, 
we use the same notation $H_\ast ( \cdot )$ for 
the homology of both topological spaces and groups unless 
otherwise stated. 
We refer to Brown's book \cite{br} 
for generalities of group (co)homology. 

\section{Lagrangian mapping class groups}\label{sec:LagMCG}

By using the ordered basis 
$\{x_1, x_2, \ldots, x_g, y_1, y_2, \ldots, y_g \}$ of $H$, 
we fix an isomorphism between $\Z^{2g}$ and $H$, 
which enables us to identify the symplectic group $\Symp{Z}$ with 
the group of automorphisms of $H$ 
preserving the intersection pairing $\mu$. Then 
the action of $\Mg$ on $H$ gives the exact sequence 
\begin{equation}\label{seq:mg}
1 \longrightarrow \Ig \longrightarrow \Mg 
\stackrel{\sigma}{\longrightarrow} \Symp{Z} \longrightarrow 1
\end{equation}
\noindent
with $\Ker \sigma = \Ig$, the Torelli group. 
The symplecticity condition for a $(2g) \times (2g)$ 
matrix $X={\small \begin{pmatrix}
A & B \\ C & D \end{pmatrix}}$ with $g \times g$ matrices $A,B,C,D$ 
is given by 
\[{}^t X \begin{pmatrix}
O & I_g \\ -I_g & O \end{pmatrix} X = \begin{pmatrix}
O & I_g \\ -I_g & O \end{pmatrix},\]
where we denote by $I_g$ the identity matrix of size $g$. 
The left hand side is equal to 
\[\begin{pmatrix}
-{}^t C A + {}^t A C & -{}^t C B + {}^t A D \\
-{}^t D A + {}^t B C & -{}^t D B + {}^t B D
\end{pmatrix}.\]
From this we see that if $C=O$, then $D = {}^t A^{-1}$ holds and 
$A^{-1}B$ is symmetric. 
This case corresponds to $\sigma(\Lg)$. 
That is, if we put 
\[urSp(2g):=\left\{{\small \begin{pmatrix}
A & B \\ O & {}^t A^{-1} \end{pmatrix}} \ \bigg| \ 
\mbox{$A^{-1}B$: symmetric} \right\}, \]
then it is a subgroup of $\Symp{Z}$ and the equality 
$\Lg=\sigma^{-1}(urSp(2g))$ follows by definition. The notation 
$urSp(2g)$ meaning ``upper right'' was introduced by Hirose in \cite{hirose}. 
We have the exact sequence
\begin{equation}\label{seq:ur}
1 \longrightarrow \Ig \longrightarrow \Lg 
\xrightarrow{\sigma|_{\Lg}}
urSp(2g) \longrightarrow 1.
\end{equation}
Moreover, if $C=O$ and $A=D=I_g$, then the matrix $B$ itself 
is symmetric. 
In this case, the subgroup 
\[\left\{{\small \begin{pmatrix}
I_g & B \\ O & I_g \end{pmatrix}} 
\ \bigg| \ \mbox{$B$: symmetric}\right\}\]
is naturally isomorphic to 
the second symmetric power $S^2 L$ of $L$ because
\[\Hom (L_y, L) \cong \Hom(L^\ast , L) \cong L \otimes L\]
and $B$ is symmetric. By definition, the equality 
$\ILg=\sigma^{-1}(S^2 L)$ holds and 
we have the exact sequence
\begin{equation}\label{seq:S2L}
1 \longrightarrow \Ig \longrightarrow \ILg 
\xrightarrow{\sigma|_{\ILg}}
S^2 L \longrightarrow 1.
\end{equation}
Note that $S^2 L$ is a free abelian group. The groups 
$S^2 L$ and $urSp(2g)$ are related by the exact sequence
\begin{equation}\label{seq:S2ur}
1 \longrightarrow S^2 L \longrightarrow urSp(2g) 
\stackrel{ul}{\longrightarrow} GL(g,\Z) \longrightarrow 1,
\end{equation}
where the map $ul$ assigns to each matrix 
its upper left block of size $g \times g$. Note that 
this group extension has a splitting defined by 
\[GL(g,\Z) \longrightarrow urSp(2g) \qquad 
\left(A \longmapsto 
{\small \begin{pmatrix} A & O \\ O & {}^tA^{-1} \end{pmatrix}}\right).\] 
Using (\ref{seq:S2ur}), we obtain the exact sequence 
\begin{equation}\label{seq:GL}
1 \longrightarrow \ILg \longrightarrow \Lg 
\xrightarrow{ul \, \circ \, \sigma |_{\Lg}}
GL(g,\Z) \longrightarrow 1.
\end{equation}

In the subsequent sections, we will use the above exact sequences 
to discuss the homology of $\Lg$ and $\ILg$. By 
a technical reason, however, we first consider 
the mapping class group $\Mgb$ 
of the surface $\Sgb$ obtained from $\Sg$ by removing 
an open disk, where each mapping class is supposed to fix 
the boundary of $\Sgb$ pointwise. The subgroups 
$\Lgb$, $\ILgb$ and $\Igb$ are defined similarly. 
Exact sequences similar to the above 
hold for these groups. 
We naturally identify $H$ with $H_1(\Sigma_{g,1})$. 
Also, we assume that $g \ge 3$ to avoid 
the complexity of $\mathcal{I}_{2,1}$, which is not covered 
by Johnson's work (see the next section).

\section{The first homology of $\ILgb$}\label{sec:H1ILgb}

We begin our investigation by determining the first homology, 
namely the abelianization, of $\ILgb$. For that,  we use 
the five-term exact sequence 
\begin{equation}\label{seq:ILgb}
H_2 (\ILgb) \to H_2 (S^2 L) \to H_1 (\Igb)_{S^2 L} 
\to H_1 (\ILgb) \to H_1(S^2 L) \to 0
\end{equation}
\noindent
associated with the group extension (\ref{seq:S2L}). 
Put
\[X_i^2:= x_i \otimes x_i, \qquad 
X_{ij}=X_{ji}:= x_i \otimes x_j + x_j \otimes x_i.\]
The set 
\[\{X_i^2 \mid 1 \le i \le g\} \cup 
\{X_{ij} \mid 1 \le i < j \le g\}\]
forms a basis of $S^2 L$ in $L \otimes L$. 
As a subgroup of $\Symp{Z}$, the group $S^2 L$ acts on $H$ by 
\begin{equation}\label{eq:action}
X_i^2 : \left\{ 
\begin{array}{l}
x_k \mapsto x_k \\
y_k \mapsto \delta_{ik} x_i +y_k
\end{array}
\right. , \qquad 
X_{ij} : \left\{ 
\begin{array}{l}
x_k \mapsto x_k \\
y_k \mapsto \delta_{jk} x_i + \delta_{ik} x_j +y_k
\end{array} \right. ,
\end{equation}
where $\delta_{ij}$ is the Kronecker delta. 
\begin{lem}\label{lem:H2ILgb}
The homomorphism 
$(\sigma|_{\ILgb})_\ast : H_2 (\ILgb) \to H_2 (S^2 L) \cong 
\wedge^2 (S^2 L)$ is surjective.
\end{lem}
\begin{proof}
We use the technique of {\it abelian cycles} to construct homology classes 
in $\Im (\sigma|_{\ILgb})_\ast$. That is, for each homomorphism 
$\varphi : \Z^2 \to \ILgb$, we have a homology class 
$\varphi_\ast (1) \in H_2 (\ILgb)$ by sending the fundamental class 
$1 \in H_2 (\Z^2) \cong \Z$ to $H_2 (\ILgb)$. 
Such a class $\varphi_\ast (1)$, which is in fact defined on cycle level, 
is called an {\it abelian cycle associated with} $\varphi$. 
Moreover, we can see that 
\[(\sigma|_{\ILgb} \circ \varphi)_\ast (1) = 
(\sigma|_{\ILgb} \circ \varphi) ((1,0)) \wedge 
(\sigma|_{\ILgb} \circ \varphi) ((0,1)) 
\in \wedge^2 (S^2 L) \cong H_2 (S^2 L),\]
where $(1,0),(0,1) \in \Z^2$  (see 
\cite[Lemma 2.2]{sa} for details). 

\begin{figure}[htbp]
\begin{center}
\includegraphics[width=0.35\textwidth]{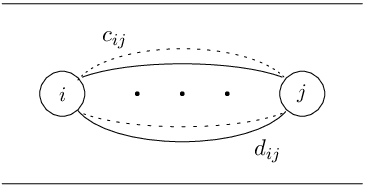}
\hskip 35pt
\includegraphics[width=0.35\textwidth]{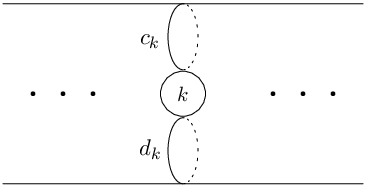}
\end{center}
\caption{}
\label{fig:s2l}
\end{figure}

Define simple closed curves $c_{ij}, d_{ij}, c_k, d_k$ $(1 \le i<j \le g, 1 \le k \le g)$ on $\Sgb$ 
as in Figure \ref{fig:s2l}.  
Let $G_1$ (resp.~$G_2$) be 
the subgroup of $\ILg$ generated by 
$\{T_{c_{ij}}\}_{i,j} \cup \{T_{c_k}\}_k$ 
(resp.~$\{T_{d_{ij}}\}_{i,j} \cup \{T_{d_k}\}_k$), where 
$T_c$ denotes the right-handed Dehn twist along a simple closed curve $c$.  
Since 
\[\sigma|_{\ILgb}(T_{c_{ij}})=\sigma|_{\ILgb}(T_{d_{ij}})=
X_i^2 -X_{ij} +X_j^2, \qquad 
\sigma|_{\ILgb}(T_{c_{k}})=\sigma|_{\ILgb}(T_{d_{k}}) 
=X_k^2,\]
each of $\sigma|_{\ILgb} (G_1)$ and $\sigma|_{\ILgb} (G_2)$ generates 
$S^2 L$. Clearly $fg=gf \in \ILg$ holds for 
any $f \in G_1$ and $g \in G_2$. Hence, for each element 
of the form $a \wedge b$ in $\wedge^2 (S^2 L)$, we can take 
$f_1 \in G_1$ and $f_2 \in G_2$ satisfying
\[ (\sigma|_{\ILgb})(f_1)=a, \quad (\sigma|_{\ILgb})(f_2)=b, \quad  
f_1 f_2 = f_2 f_1.\] 
They give a homomorphism 
$\varphi: \Z^2 \to \ILgb$ with 
$(\sigma|_{\ILgb} \circ \varphi)_\ast (1) = a \wedge b$, 
which implies the surjectivity of $(\sigma|_{\ILgb})_\ast$.
\end{proof}
\noindent
Lemma \ref{lem:H2ILgb} shows that $H_2 (\ILgb)$ 
is non-trivial (see also Theorem \ref{thm:pure}). 
In particular, its rank, which may be infinite, 
gets bigger and bigger when $g$ grows. 

Before going further, here we recall some results on 
the Torelli group $\Igb$ obtained by Johnson 
in \cite{jo}--\cite{jo3}. 
First, he showed in \cite{jo1} that $\Igb$ is finitely generated 
for $g \ge 3$. 
This fact together with the sequences (\ref{seq:ur}), (\ref{seq:S2L}) 
imply that $\Lgb$ and $\ILgb$ are also finitely generated. 
At present, it is not known whether they are finitely presentable or not, 
where the same question for $\Igb$ is a well-known open problem. 
Second, he showed that $\Igb$ is normally 
generated by only one element $T_{c_2} T_{d_2}^{-1}$ 
(see Figure \ref{fig:s2l}). Finally, in \cite{jo3}, 
he determined the abelianization of $\Igb$ written as follows. 
Let $B$ be a commutative $\Z_2$-algebra with unit $1$ 
generated by formal elements $\overline{x}$ for 
$x \in H \otimes \Z_2$ and having relations 
\[\overline{x}^2 = \overline{x}, \qquad 
\overline{x+y} = \overline{x} + \overline{y} + \overline{\mu}(x,y)\]
for $x,y \in H \otimes \Z_2$, where 
$\overline{\mu}(x,y):= \mu (x,y) \bmod 2$. 
The algebra $B$ can be graded by supposing that each 
$\overline{x}$ has degree $1$ (after replacing $\overline{x}^2$ 
by $\overline{x}$). 
Let $B^i$ be the submodule of $B$ generated by elements of 
degree at most $i$. 
This endows $B$ with a filtration
\[B^3 \supset B^2 \supset B^1 \supset B^0=\{0,1\}.\] 
We have a natural action of $\Mgb$ on $B^3$ defined by 
$f \, \overline{x} := \overline{f_\ast (x)}$. It is easily checked that 
there exists a natural $\Mgb$-equivariant isomorphism
\[B^3/B^2 \cong \wedge^3 (H \otimes \Z_2).\]
Therefore we can take the fiber product 
$\wedge^3 H \times_{\wedge^3 (H \otimes \Z_2)} B^3$ 
of the natural projections 
$B^3 \to B^3/B^2 \cong \wedge^3 (H \otimes \Z_2)$ and 
$\wedge^3 H \to \wedge^3 (H \otimes \Z_2)$. 
Then Johnson gave an $\Mgb$-equivariant isomorphism 
\[(\tau,\beta): H_1 (\Igb) \stackrel{\cong}{\longrightarrow} 
\wedge^3 H \times_{\wedge^3 (H \otimes \Z_2)} B^3,\]
where $\Mgb$ acts on $\Igb$ and $H_1 (\Igb)$ by conjugation and 
on $\wedge^3 H \times_{\wedge^3 (H \otimes \Z_2)} B^3$ diagonally. 
The homomorphism $\tau$ is now called the {\it Johnson homomorphism} 
\cite{jo, jo2} and 
$\beta$ is called the {\it Birman-Craggs-Johnson homomorphism} 
(see \cite{joq} and Birman-Craggs \cite{bc}). 
Explicitly, the isomorphism is given by 
\[T_{c_2} T_{d_2}^{-1} \longmapsto (x_1 \wedge y_1 \wedge y_2, 
\overline{x}_1 \overline{y}_1 (\overline{y}_2+1)),\]
which characterizes an $\Mgb$-equivariant homomorphism uniquely 
because $\Igb$ is normally generated by $T_{c_2} T_{d_2}^{-1}$.

\begin{lem}\label{lem:Ig_Coinv}
$H_1 (\Igb)_{S^2 L} \cong 
\left\{\begin{array}{ll}
\wedge^3 L^\ast \oplus L^\ast \oplus 
\wedge^2 (L^\ast \otimes \Z_2) & 
(g = 3), \\
\wedge^3 L^\ast \oplus L^\ast & 
(g \ge 4).
\end{array}\right.$
\end{lem}
\begin{proof}
By definition, the coinvariant part $H_1 (\Igb)_{S^2 L}$ is the quotient 
of $H_1 (\Igb)$ 
by the submodule $Q_0$ generated by $\{ \sigma x-x \mid 
\sigma \in S^2 L, x \in H_1 (\Igb)\}$. We now list 
a generating set of $Q_0$ explicitly. Assuming that 
the indices $i,j,k,l \in \{1,2,\ldots,g\}$ are distinct from each other, we 
have 
\begin{align*}
&
X_j^2(x_i \wedge x_j \wedge y_j, \overline{x}_i \overline{x}_j\overline{y}_j)
-(x_i \wedge x_j \wedge y_j, \overline{x}_i \overline{x}_j\overline{y}_j) \\
=\ & (x_i \wedge x_j \wedge (x_j+y_j), \overline{x}_i \overline{x}_j 
\overline{x_j+y_j}) 
-(x_i \wedge x_j \wedge y_j, \overline{x}_i \overline{x}_j\overline{y}_j) \\
=\ & (x_i \wedge x_j \wedge y_j, \overline{x}_i \overline{x}_j 
(\overline{x}_j+\overline{y}_j+1)) 
-(x_i \wedge x_j \wedge y_j, \overline{x}_i \overline{x}_j\overline{y}_j) \\
=\ & (0, \overline{x}_i \overline{x}_j^2+ 
\overline{x}_i \overline{x}_j) = 
(0,0), 
\end{align*}
\noindent
where we used the relations 
$\overline{x}_j^2=\overline{x}_j$ 
and $2 \overline{x}_i \overline{x}_j=0$ in $B^3$. 
We denote this result by
\begin{align*}
&(1a) \qquad 
[X_j^2; (x_i \wedge x_j \wedge y_j, \overline{x}_i \overline{x}_j\overline{y}_j)]
:= (0,0), \\
\intertext{for short. Similar calculations show that}
&(1b) \qquad 
[X_{kj};(x_i \wedge x_j \wedge y_j, 
\overline{x}_i \overline{x}_j\overline{y}_j)] = 
(x_i \wedge x_j \wedge x_k, 
\overline{x}_i \overline{x}_j \overline{x}_k), \\
&(1c) \qquad 
[X_{ij};(x_i \wedge x_j \wedge y_j, 
\overline{x}_i \overline{x}_j\overline{y}_j)] = 
(0, \overline{x}_i \overline{x}_j), \\
&(2a) \qquad 
[X_k^2; (x_i \wedge x_j \wedge y_k, \overline{x}_i \overline{x}_j\overline{y}_k)]
= (x_i \wedge x_j \wedge x_k, \overline{x}_i \overline{x}_j \overline{x}_k+
\overline{x}_i \overline{x}_j), \\
&(2b) \qquad 
[X_{jk}; (x_i \wedge x_j \wedge y_k, 
\overline{x}_i \overline{x}_j\overline{y}_k)]
= (0, \overline{x}_i \overline{x}_j), \\
&(2c)^\ast \qquad 
[X_{lk}; (x_i \wedge x_j \wedge y_k, 
\overline{x}_i \overline{x}_j\overline{y}_k)]
= (x_i \wedge x_j \wedge x_l, \overline{x}_i \overline{x}_j \overline{x}_l), \\
&(3a) \qquad 
[X_j^2; (x_i \wedge y_i \wedge y_j, 
\overline{x}_i \overline{y}_i \overline{y}_j)]
= (-x_i \wedge x_j \wedge y_i, 
\overline{x}_i \overline{x}_j \overline{y}_i+\overline{x}_i \overline{y}_i), \\
&(3b) \qquad 
[X_{ik}; (x_i \wedge y_i \wedge y_j, 
\overline{x}_i \overline{y}_i \overline{y}_j)]
= (x_i \wedge x_k \wedge y_j, \overline{x}_i \overline{x}_k \overline{y}_j), \\
&(3c) \qquad 
[X_{jk}; (x_i \wedge y_i \wedge y_j, 
\overline{x}_i \overline{y}_i \overline{y}_j)]
= (-x_i \wedge x_k \wedge y_i, 
\overline{x}_i \overline{x}_k \overline{y}_i), \\
&(3d) \qquad 
[X_{ij}; (x_i \wedge y_i \wedge y_j, 
\overline{x}_i \overline{y}_i \overline{y}_j)]
= (x_i \wedge x_j \wedge y_j, \overline{x}_i \overline{x}_j \overline{y}_j+
\overline{x}_i \overline{x}_j+\overline{x}_i \overline{y}_i), \\
&(4a) \qquad 
[X_j^2; (x_i \wedge y_j \wedge y_k, 
\overline{x}_i \overline{y}_j \overline{y}_k)]
= (x_i \wedge x_j \wedge y_k, 
\overline{x}_i \overline{x}_j \overline{y}_k+\overline{x}_i \overline{y}_k), \\
&(4b) \qquad 
[X_{ij}; (x_i \wedge y_j \wedge y_k, 
\overline{x}_i \overline{y}_j \overline{y}_k)]
= (0, \overline{x}_i \overline{y}_k), \\
&(4c)^\ast \qquad 
[X_{jl}; (x_i \wedge y_j \wedge y_k, 
\overline{x}_i \overline{y}_j \overline{y}_k)]
= (x_i \wedge x_l \wedge y_k, 
\overline{x}_i \overline{x}_l \overline{y}_k), \\
&(4d) \qquad 
[X_{jk}; (x_i \wedge y_j \wedge y_k, 
\overline{x}_i \overline{y}_j \overline{y}_k)]
= (x_i \wedge x_k \wedge x_j+x_i \wedge x_k \wedge y_k 
-x_i \wedge x_j \wedge y_j, \\
&\hskip 205pt \overline{x}_i \overline{x}_k \overline{x}_j
+\overline{x}_i \overline{x}_k \overline{y}_k
+\overline{x}_i \overline{x}_j \overline{y}_j), \\
&(5a) \qquad 
[X_i^2; (y_i \wedge y_j \wedge y_k, 
\overline{y}_i \overline{y}_j \overline{y}_k)]
= (x_i \wedge y_j \wedge y_k, 
\overline{x}_i \overline{y}_j \overline{y}_k+\overline{y}_j \overline{y}_k), \\
&(5b)^\ast \qquad 
[X_{il}; (y_i \wedge y_j \wedge y_k, 
\overline{y}_i \overline{y}_j \overline{y}_k)]
= (x_l \wedge y_j \wedge y_k, 
\overline{x}_l \overline{y}_j \overline{y}_k), \\
&(5c) \qquad 
[X_{ij}; (y_i \wedge y_j \wedge y_k, 
\overline{y}_i \overline{y}_j \overline{y}_k)]
= (x_j \wedge x_i \wedge y_k+x_j \wedge y_j \wedge y_k 
-x_i \wedge y_i \wedge y_k, \\
&\hskip 205pt \overline{x}_j \overline{x}_i \overline{y}_k
+\overline{x}_j \overline{y}_j \overline{y}_k
+\overline{x}_i \overline{y}_i \overline{y}_k), \\
&(6) \ \ \qquad 
[X_{ij}; (0, \overline{x}_i \overline{y}_i)]
= (0, \overline{x}_i \overline{x}_j), \\
&(7a) \qquad 
[X_j^2; (0, \overline{x}_i \overline{y}_j)]
= (0, \overline{x}_i \overline{x}_j +\overline{x}_i), \\
&(7b) \qquad 
[X_{ij}; (0, \overline{x}_i \overline{y}_j)]
= (0, \overline{x}_i), \\
&(7c) \qquad 
[X_{jk}; (0, \overline{x}_i \overline{y}_j)]
= (0, \overline{x}_i \overline{x}_k), \\
&(8a) \qquad 
[X_i^2; (0, \overline{y}_i \overline{y}_j)]
= (0, \overline{x}_i \overline{y}_j +\overline{y}_j), \\
&(8b) \qquad 
[X_{ik}; (0, \overline{y}_i \overline{y}_j)]
= (0, \overline{x}_k \overline{y}_j), \\
&(8c) \qquad 
[X_{ij}; (0, \overline{y}_i \overline{y}_j)]
= (0, \overline{x}_i \overline{x}_j +\overline{x}_i\overline{y}_i 
+\overline{x}_j\overline{y}_j), \\
&(9a) \qquad 
[X_i^2; (0, \overline{y}_i )]
= (0, \overline{x}_i+1), \\
&(9b) \qquad 
[X_{ij}; (0, \overline{y}_i )]
= (0, \overline{x}_j), 
\end{align*}

\noindent
where $( \ \cdot \ )^\ast$ means that it is valid for $g \ge 4$. 
The actions not listed above are all trivial, 
namely $\sigma x - x=(0,0)$, so that they do not contribute to $Q_0$. 
In particular, there are no contribution from the elements
\[(x_i \wedge x_j \wedge x_k, \overline{x}_i \overline{x}_j\overline{x}_k), 
\quad (0 , \overline{x}_i \overline{x}_j), \quad 
(0, \overline{x}_i), \quad (0, 1).\]

From $(7b, 9a, 1c, 1b, 4b, 8a, 3b, 3c, 3a, 5b, 5a)$, 
we see that, for $g \ge 4$, $Q_0$ contains 
$(0,\overline{x}_j)$, $(0,1)$, $(0, \overline{x}_i \overline{x}_j)$, 
$(x_i \wedge x_j \wedge x_k, 
\overline{x}_i \overline{x}_j \overline{x}_k)$, 
$(0, \overline{x}_i \overline{y}_k)$, 
$(0, \overline{y}_j)$, $(x_i \wedge x_k \wedge y_j, 
\overline{x}_i \overline{x}_k \overline{y}_j)$, 
$(x_i \wedge x_k \wedge y_i, 
\overline{x}_i \overline{x}_k \overline{y}_i)$, 
$(0, \overline{x}_i \overline{y}_i)$, 
$(x_l \wedge y_j \wedge y_k, 
\overline{x}_l \overline{y}_j \overline{y}_k)$, 
$(0, \overline{y}_j \overline{y}_k)$ in order, 
and that combinations of these elements 
express all the generators listed above except $(5c)$. Finally 
$(5c)$ shows that $(x_j \wedge y_j \wedge y_k 
-x_i \wedge y_i \wedge y_k, \overline{x}_j \overline{y}_j \overline{y}_k
+\overline{x}_i \overline{y}_i \overline{y}_k)$ are in $Q_0$. 
Our claim for $g \ge 4$ follows from this, where 
we assign $y_k \in L^\ast$ to $(y_k \wedge x_i \wedge y_i, 
\overline{y}_k \overline{x}_i \overline{y}_i) \in H_1 (\Igb)_{S^2 L}$, 
which does not depend on $i$. 

When $g=3$, differently from the above, we cannot remove 
$(x_l \wedge y_j \wedge y_k, 
\overline{x}_l \overline{y}_j \overline{y}_k)$ and 
$(0, \overline{y}_j \overline{y}_k)$ simultaneously. 
In this case, we use $(5a)$ to 
eliminate $(x_l \wedge y_j \wedge y_k, 
\overline{x}_l \overline{y}_j \overline{y}_k)$ and 
conclude that 
$(0, \overline{y}_j \overline{y}_k)$ survive in 
$H_1 (\Igb)_{S^2 L}$ and 
form $\wedge^2 (L^\ast \otimes \Z_2)$. 
\end{proof}
By the exact sequence (\ref{seq:ILgb}) together with Lemmas \ref{lem:H2ILgb}, 
\ref{lem:Ig_Coinv}, we conclude the following. 
\begin{thm}\label{thm:H1ILgb}
$H_1 (\ILgb) \cong \left\{\begin{array}{ll}
\wedge^3 L^\ast \oplus L^\ast \oplus 
\wedge^2 (L^\ast \otimes \Z_2) \oplus S^2 L & 
(g = 3), \\
\wedge^3 L^\ast \oplus L^\ast \oplus S^2 L & 
(g \ge 4).\end{array}\right.$
\end{thm}

\begin{remark}
In \cite[Theorem 1]{le1}, 
Levine constructed a surjective homomorphism 
\[\mathcal{J}: 
H_1 (\ILgb) \twoheadrightarrow \wedge^3 L^\ast \oplus L^\ast\] 
by using the Johnson homomorphism $\tau$ for $\Igb$. We can check that 
$\mathcal{J}$ coincides with the projection to the first two components of 
the isomorphism in Theorem \ref{thm:H1ILgb}. 
In \cite[Section 5.1]{bfp}, 
Broaddus-Farb-Putman gave another construction of $\mathcal{J}$. 
In fact, their homomorphisms called {\it relative Johnson homomorphisms} 
cover not only $\ILgb$ but any subgroup of $\Mgb$ fixing a given 
submodule of $H$. 
\end{remark}

\section{The first and second homology of $\ur{2g}$}\label{sec:H_ur}

In this section, we determine the first and second homology 
of $\ur{2g}$ for later use. By a technical reason, we first consider 
its index $2$ subgroup $\urp{2g}$ defined by 
\begin{equation}\label{seq:urSp+}
1 \longrightarrow \urp{2g} \longrightarrow \ur{2g} 
\xrightarrow{\det {\tiny \circ} ul} \Z_2 \longrightarrow 1. 
\end{equation}
\noindent
By restricting the sequence (\ref{seq:S2ur}) to $\urp{2g}$, we have 
a split exact sequence 
\begin{equation}\label{seq:S2ur+}
1 \longrightarrow S^2 L \longrightarrow \urp{2g} 
\stackrel{ul}{\longrightarrow} SL(g,\Z) \longrightarrow 1.
\end{equation}
\begin{prop}\label{lem:urSp+}
\begin{tabular}[t]{ll}
$(1)$ & The group $\urp{2g}$ is perfect, that is $H_1(\urp{2g}) = 0$ 
for $g \ge 3$. \\ 
&\\
$(2)$ & $H_2 (\urp{2g}) \cong \left\{\begin{array}{ll}
\Z_2 \oplus \Z_2 \oplus \Z_2 \oplus \Z_2 & (g = 3), \\
\Z_2 \oplus \Z_2 & (g = 4), \\
\Z_2 & (g \ge 5).
\end{array}\right.$
\end{tabular}
\end{prop}
We will prove this proposition by using the Lyndon-Hochschild-Serre 
spectral sequence 
\begin{equation}\label{LHSseq1}
E_{p,q}^2 =H_p(SL(g,\Z);H_q(S^2 L)) \Longrightarrow 
H_n (\urp{2g})
\end{equation}
\noindent
associated with (\ref{seq:S2ur+}). 
Before that, 
we recall the first and second homology of $SL(g,\Z)$. 
We refer to books of Milnor \cite[Sections 5 and 10]{milnorK} 
and Rosenberg \cite[Sections 4.1 and 4.2]{rosenberg} 
for the facts below and generalities of the second homology of groups. 
The group $SL(g,\Z)$ has a presentation 
\begin{center}
\begin{tabular}{ll}
\textbullet \ generators: & $\{e_{ij} \mid 
\mbox{$1 \le i \le g, 1 \le j \le g $ and $i \neq j$}\}$,\\ 
\textbullet \ relations: &
$\begin{array}[t]{ll}
{[}e_{ij},e_{kl}]=1 & \mbox{if $j \neq k$ and $i \neq l$},\\
{[}e_{ik},e_{kj}]=e_{ij} & \mbox{if $i \neq j \neq k \neq i$},\\
(e_{12}e_{21}^{-1}e_{12})^4=1, &
\end{array}$
\end{tabular}
\end{center}
where $e_{ij}$ corresponds to 
the matrix whose diagonal entries and $(i,j)$-entry are $1$ with 
the others $0$. 
From this presentation, we immediately see that $SL(g,\Z)$ 
is perfect 
for every $g \ge 3$. The second homology, which is also called 
the {\it Schur multiplier}, of 
$SL(g,\Z)$ is also known:
\[H_2 (SL(g,\Z)) \cong \left\{\begin{array}{ll}
\Z_2 \oplus \Z_2 & (g = 3,4, \ \mbox{by van der Kallen \cite{vd}}), \\
\Z_2 & (g \ge 5),
\end{array}\right.\] 
where van der Kallen also showed in \cite{vd} that 
one summand of $H_2 (SL(3,\Z)) \cong \Z_2 \oplus \Z_2$ survives in 
the stable homology 
$\displaystyle\lim_{g \to \infty} H_2(SL(g,\Z)) 
\cong K_2(\Z) \cong \Z_2$ under stabilization, 
while the other one vanishes in $H_2 (SL(4,\Z))$. 

For computations of the zeroth and first homology of a group $G$, 
we can use any connected CW-complex $X$ with $\pi_1 X=G$. 
Let $X_g$ be a connected CW-complex associated with the above 
presentation of $SL(g,\Z)$. Namely $X_g$ consists 
of one vertex, edges $\{\langle e_{ij} \rangle \mid 
\mbox{$1 \le i \le g, 1 \le j \le g $ and $i \neq j$}\}$ and 
faces 
\[\{\langle [e_{ij},e_{kl}]\rangle \mid 
\mbox{$j \neq k$ and $i \neq l$}\} \cup 
\{\langle[e_{ik},e_{kj}]e_{ij}^{-1}\rangle \mid 
\mbox{if $i \neq j \neq k \neq i$}\} \cup
\{\langle(e_{12}e_{21}^{-1}e_{12})^4\rangle\}\] 
attached to the 1-skeleton of $X_g$ along the words. 
We consider $S^2 L$ to be a local coefficient system on $X_g$. 
The boundary maps 
\begin{align*}
&\partial_1 : C_1 (X_g;S^2 L) \to C_0 (X_g;S^2 L) \cong S^2 L, \\
&\partial_2 : C_2 (X_g;S^2 L) \to C_1 (X_g;S^2 L)
\end{align*}
\noindent
of the complex $C_\ast (X_g;S^2 L) = C_\ast (X_g) \otimes S^2 L$ 
are given by 
\begin{align*}
\partial_1(\langle e_{ij} \rangle \otimes c)&= (e_{ij}^{-1}-1)c,\\
\partial_2(\langle e_1 e_2 \cdots e_n \rangle \otimes c)&= 
\langle e_1 \rangle \otimes c + \langle e_2 \rangle \otimes 
e_1^{-1} c + \langle e_3 \rangle \otimes (e_1e_2)^{-1}c+\cdots \\
& \quad + \langle e_n \rangle \otimes (e_1 e_2 \cdots e_{n-1})^{-1} c
\end{align*}
\noindent
for $c \in S^2 L$, where $e_1, e_2, \ldots, e_n \in 
\{e_{ij}\}_{i,j} \cup \{e_{ij}^{-1}\}_{i,j}$ 
and $\langle e_{ij}^{-1} \rangle \otimes c := 
- \langle e_{ij} \rangle \otimes e_{ij} c$. The action of 
$SL(g,\Z)$ on $L$ is given by 
\[e_{ij}: x_k \mapsto \delta_{jk} x_i +x_k, 
\qquad 
e_{ij}^{-1}: x_k \mapsto -\delta_{jk} x_i +x_k.\]
\begin{lem}\label{lem:SLS2L}
\begin{tabular}[t]{ll}
$(1)$ & $H_0 (SL(g,\Z);S^2 L) \cong (S^2 L)_{SL(g,\Z)} 
= 0$ \ for $g \ge 3$. \\
$(2)$ & $H_1 (SL(g,\Z);S^2 L) = 0$ \ for $g \ge 4$. 
\end{tabular}
\end{lem}
\begin{proof}
Here and hereafter, we suppose that the indices $i,j,k,l$ are distinct from 
each other. We have 
\begin{align*}
&\partial_1 (\langle e_{ij} \rangle \otimes X_j^2) =
(e_{ij}^{-1}-1)X_j^2=(-x_i+x_j)^{\otimes 2}-X_j^2=X_i^2-X_{ij},\\
&\partial_1 (\langle e_{ij} \rangle \otimes X_{jk}) =
(e_{ij}^{-1}-1)X_{jk}=(-X_{ik}+X_{jk})-X_{jk}=-X_{ik}. 
\end{align*}
\noindent
By running $i,j,k$ in $\{1,2,\ldots,g\}$ with $g \ge 3$, we 
immediately see that $\partial_1$ is surjective and $(1)$ holds. 
To show $(2)$, it suffices to check that 
$\partial_1: C_1(SL(g,\Z);S^2 L)/\Im \partial_2 \to S^2 L$ is 
an isomorphism. 
Assume that $g \ge 4$. $C_1 (SL(g,\Z);S^2 L)$ is generated 
by elements of types 
\[\begin{array}{llll}
\mathrm{I}: \ \langle e_{ij} \rangle \otimes X_i^2, & 
\mathrm{II}: \ \langle e_{ij} \rangle \otimes X_j^2, & 
\mathrm{III}: \ \langle e_{ij} \rangle \otimes X_k^2, & \\
\mathrm{IV}: \ \langle e_{ij} \rangle \otimes X_{ij}, & 
\mathrm{V}: \ \langle e_{ij} \rangle \otimes X_{jk}, &
\mathrm{VI}: \ \langle e_{ij} \rangle \otimes X_{il}, &
\mathrm{VII}: \langle e_{ij} \rangle \otimes X_{kl}.
\end{array}\]
For $c \in S^2 L$, we have  
\[\partial_2(\langle [e_{ik},e_{kj}] e_{ij}^{-1} \rangle \otimes c)
=\langle e_{ik} \rangle \otimes (1-e_{kj}^{-1}e_{ij}^{-1})c
+\langle e_{kj} \rangle \otimes (e_{ik}^{-1}-e_{ij}^{-1})c
-\langle e_{ij} \rangle \otimes c.\]
By putting $c=X_i^2, X_j^2, X_k^2, X_l^2, X_{jk}$ and $X_{jl}$, we see that 
\[\begin{array}{cl}
\mathrm{(i)} & -\langle e_{ij} \rangle \otimes X_i^2 \quad (\mbox{type I}),\\
\mathrm{(ii)} &\langle e_{ik} \rangle \otimes 
(X_{ij}+X_{jk}-X_{ik}-X_i^2-X_k^2)
+\langle e_{kj} \rangle \otimes (-X_i^2+X_{ij}) 
-\langle e_{ij} \rangle \otimes X_j^2, \\
\mathrm{(iii)} &\langle e_{kj} \rangle \otimes (X_i^2-X_{ik}+X_k^2)
-\langle e_{ij} \rangle \otimes X_k^2, \\
\mathrm{(iv)} &-\langle e_{ij} \rangle \otimes X_l^2 
\quad (\mbox{type III}),\\
\mathrm{(v)} & \langle e_{ik} \rangle \otimes (X_{ik}+2X_k^2)
+\langle e_{kj} \rangle \otimes (X_{ik}-X_{ij}) 
-\langle e_{ij} \rangle \otimes X_{jk},\\
\mathrm{(vi)} & \langle e_{ik} \rangle \otimes (X_{il}+X_{kl})
+\langle e_{kj} \rangle \otimes X_{il}
-\langle e_{ij} \rangle \otimes X_{jl}
\end{array}\]
are in $\Im \partial_2$. From (i), (iii) and (iv), 
$-\langle e_{kj} \rangle \otimes X_{ik}$ (type VI) is in $\Im \partial_2$. 
Also 
\begin{align*}
\mathrm{(vii)} \quad 
\partial_2(\langle [e_{ij},e_{kl}] \rangle \otimes X_l^2)
&=\langle e_{ij} \rangle \otimes (1-e_{kl}^{-1})X_l^2
+\langle e_{kl} \rangle \otimes (e_{ij}^{-1}-1)X_l^2 \\
&=\langle e_{ij} \rangle \otimes (X_{kl}-X_k^2), \\
\mathrm{(viii)} \quad 
\partial_2(\langle [e_{ij},e_{kj}] \rangle \otimes X_j^2)
&=\langle e_{ij} \rangle \otimes (-X_k^2+X_{kj})
+\langle e_{kj} \rangle \otimes (X_i^2-X_{ij})
\end{align*}
\noindent
are in $\Im \partial_2$. From (iv) and (vii), 
$\langle e_{ij} \rangle \otimes X_{kl}$ (type VII) is in $\Im \partial_2$. 
Then we can derive from (vi) that 
\begin{align*}
\mathrm{(ix)} &\quad \langle e_{ik} \rangle \otimes X_{kl} 
-\langle e_{ij} \rangle \otimes X_{jl} \in \Im \partial_2.\\
\intertext{We see from (iv) and (viii) that}
\mathrm{(x)} &\quad \langle e_{ij} \rangle \otimes X_{jk}
-\langle e_{kj} \rangle \otimes X_{ji} \in \Im \partial_2.\\
\intertext{Finally, we can derive from (ii) and (v) that} 
\mathrm{(xi)} &\quad \langle e_{ik} \rangle \otimes 
(X_{jk}-X_{ik}-X_k^2) + \langle e_{kj} \rangle 
\otimes X_{ij} -\langle e_{ij} \rangle \otimes X_j^2\\
\mathrm{(xii)} &\quad \langle e_{ik} \rangle 
\otimes (X_{ik}+2X_k^2) -\langle e_{kj} \rangle 
\otimes X_{ij} -\langle e_{ij} \rangle \otimes X_{jk}
\end{align*}
\noindent
are in $\Im \partial_2$. 

We have so far shown that $C_1(SL(g,\Z);S^2 L)/\Im \partial_2$ is a 
quotient of the module $M$ generated by 
the elements of types (II), (IV) and (V) with the relations 
(ix), (x), (xi) and (xii). 
We can use (xii) to remove $\langle e_{ik} \rangle 
\otimes X_{ik}$ (type IV) and to produce a relation
\[\mathrm{(xiii)} \quad \langle e_{ik} \rangle \otimes 
(X_{jk}+X_k^2) - \langle e_{ij} \rangle \otimes (X_j^2+X_{jk})\]
in $M$ from (xi). Therefore $M$ is generated by 
the elements of types (II) and (V) with the relations 
(ix), (x), (xiii). 
The relation (ix) enables us to put 
$Y_{il}:=-\langle e_{ij} \rangle \otimes X_{jl} \in M$, which 
does not depend on $j$, and the relation (x) shows that 
$Y_{il}=Y_{li}$. On the other hand, 
if we put $Y_i(j,k):= \langle e_{ij} \rangle \otimes X_j^2 - 
\langle e_{ik} \rangle \otimes X_{kj}$, 
it follows from (ix) and (xiii) that 
$Y_i (j,l) = Y_i (j,k) = Y_i (k,j) \in M$. This implies that 
$Y_i:=Y_i (j,l) \in M$ is independent of $j$ and $l$. 
Consequently, 
$M$ is a free module with a basis 
$\{Y_i \mid 1 \le i \le g\} \cup 
\{ Y_{jk} \mid 1 \le j < k \le g\}$. 
It is easy to see that the homomorphism 
$\widetilde{\partial}_1: 
M \to C_0 (SL(g,\Z);S^2 L) \cong S^2 L$ induced from 
the surjection $\partial_1:C_1 (SL(g,\Z);S^2 L)/\Im \partial_2 
\twoheadrightarrow C_0 (SL(g,\Z);S^2 L)$ is an isomorphism 
since $\widetilde{\partial}_1 (Y_i) = X_i^2$ and 
$\widetilde{\partial}_1(Y_{jk})=X_{jk}$. 
Hence $\partial_1: C_1 (SL(g,\Z);S^2 L)/\Im \partial_2 \to S^2 L$ 
is an isomorphism and (2) is proved. 
\end{proof}
\begin{lem}\label{lem:SLWS2L}
$H_0 (SL(g,\Z);H_2(S^2 L)) \cong (\wedge^2(S^2 L))_{SL(g,\Z)} = 0$ 
for $g \ge 4$. 
\end{lem}
\begin{proof}
By definition, the coinvariant part $(\wedge^2(S^2 L))_{SL(g,\Z)}$ 
is the quotient of $\wedge^2(S^2 L)$ 
by the submodule $Q_1$ generated by $\{ [e;x] \mid 
e \in SL(g,\Z), \ x \in \wedge^2(S^2 L)\}$, where we put 
$[e;x]:=e x-x$. 
Direct computations show that
\[\begin{array}{cl}
\mathrm{(i)} & [e_{ij};X_i^2 \wedge X_j^2] = X_i^2 \wedge X_{ij},\\
\mathrm{(ii)} & [e_{kl}; X_i^2 \wedge X_{jl}] = X_i^2 \wedge X_{jk},\\
\mathrm{(iii)} & [e_{kj};X_i^2 \wedge X_j^2] = X_i^2 \wedge (X_k^2 +X_{jk}),\\
\mathrm{(iv)} & [e_{ji};X_i^2 \wedge X_{jk}] 
= (X_j^2 + X_{ij}) \wedge X_{jk}, \\
\mathrm{(v)} & [e_{ij};X_j^2 \wedge X_{kl}] = X_i^2 \wedge X_{kl} 
+X_{ij} \wedge X_{kl}
\end{array}\]
are in $Q_1$ and that they generate $\wedge^2(S^2 L)$ 
by running $i,j,k,l$ in $\{1,2,\ldots,g\}$ with $g \ge 4$.  
This completes the proof.
\end{proof}
\begin{proof}[Proof of Proposition $\ref{lem:urSp+}$ $(1)$ for $g \ge 3$ 
and $(2)$ for $g \ge 4$]
When $g \ge 3$, we have $E_{1,0}^2=E_{0,1}^2=0$ in the spectral sequence (\ref{LHSseq1}) 
by Lemma \ref{lem:SLS2L} (1) and the fact that $H_1 (SL(g,\Z))=0$. 
This proves (1) . 

Assume further that $g \ge 4$. 
By Lemma \ref{lem:SLS2L} (2) and Lemma \ref{lem:SLWS2L}, 
we have $E_{1,1}^2=E_{0,2}^2=0$ 
in the spectral sequence (\ref{LHSseq1}). It follows that $H_2 (\urp{2g}) \cong H_2(SL(g,\Z))$. 
We finish the proof of (2) for $g \ge 4$ by using the explicit description of $H_2 (SL(g,\Z))$. 
\end{proof}
\begin{cor}\label{cor:urSp}
\begin{tabular}[t]{ll}
$(1)$ & $H_1 (\ur{2g}) \cong H_1 (GL(g,\Z)) \cong \Z_2$ \ for $g \ge 3$. \\
$(2)$ & $H_2 (\ur{2g}) \cong H_2 (\urp{2g})$ \ for $g \ge 3$. 
\end{tabular}
\end{cor}
\begin{proof}
By using the Lyndon-Hochschild-Serre spectral sequence 
associated with the split extension (\ref{seq:urSp+}) and 
the fact that $H_1 (\urp{2g})=0$, we have $H_1 (\ur{2g}) \cong \Z_2$ and 
$H_2 (\ur{2g}) \cong H_2 (\urp{2g})_{\Z_2}$. 
For $g \ge 4$, we have seen that $H_2 (\urp{2g})_{\Z_2} \cong 
H_2 (SL(g,\Z))_{\Z_2}$. The action of $\Z_2$ on $H_2 (\urp{2g})$ 
is compatible with that on $H_2 (SL(g,\Z))$ 
and the latter one is known to be trivial. Hence 
$H_2 (\ur{2g}) \cong H_2 (\urp{2g})$ follows. When $g=3$, 
the action of $\Z_2$ on $H_2 (\urp{2g})$ is also trivial, 
since we can take the minus of the identity matrix 
as a lift of the generator of $\Z_2$ and it is central. 
Therefore $H_2 (\ur{2g}) \cong H_2 (\urp{2g})$ holds also 
for $g = 3$.
\end{proof}
It remains to compute $H_2 (\urp{2g})$ for $g=3$. 

\begin{lem}\label{lem:genus3_1}
$H_1 (SL(3,\Z);S^2 L) \cong \Z_2$ and it is generated by 
$\langle e_{12} \rangle \otimes X_3^2$. 
\end{lem}
\begin{proof}[Sketch of Proof]
Now $SL(3,\Z)$ has a presentation consisting of $6$ generators and 
$13$ relations. Also we have $S^2 L \cong \Z^6$. Hence the complex 
\[C_2 (SL(3,\Z);S^2 L) \xrightarrow{\partial_2} C_1 (SL(3,\Z);S^2 L) 
\xrightarrow{\partial_1} C_0 (SL(3,\Z);S^2 L)\] 
can be explicitly written as $\Z^{78} \xrightarrow{D_2 \cdot} \Z^{36} 
\xrightarrow{D_1 \cdot} \Z^{6}$ with some matrices $D_1$ and $D_2$. 
The author with an aid of a computer calculated the homology 
by using the Smith normal form. 
We omit the details. 
\end{proof}
\begin{lem}\label{lem:genus3_2}
$H_0 (SL(3,\Z);H_2(S^2 L)) \cong 
H_2(S^2 L)_{SL(3,\Z)} 
\cong (\wedge^2(S^2 L))_{SL(3,\Z)} 
\cong \Z_2$ and 
it is generated by $X_3^2 \wedge X_2^2$. Moreover 
this generator is mapped non-trivially to $H_2 (Sp(6,\Z))$ by 
the composition $H_2(S^2 L)_{SL(3,\Z)} \to H_2 (\urp{6}) \to H_2 (Sp(6,\Z))$ induced 
from the inclusions $S^2 L \hookrightarrow \urp{6} \hookrightarrow Sp(6,\Z)$. 
\end{lem} 
In the proof of this lemma, the following theorem by Stein plays a key role. 
\begin{thm}[Stein {\cite[Theorem 2.2]{stein}}]\label{thm:Stein}
$H_2 (Sp(6,\Z)) \cong \Z \oplus \Z_2$ and the abelian cycle associated with the 
homomorphism $\varphi : \mathbb{Z}^2 \to Sp(6,\Z)$ defined by 
\[\varphi ((1,0))=X_3^2, \quad \varphi ((0,1))=X_2^2\]
gives the element of order $2$, where 
$X_3^2$ and $X_2^2$ are in $S^2 L \subset \urp{6} \subset Sp(6,\Z)$.
\end{thm}
\begin{proof}[Proof of Lemma $\ref{lem:genus3_2}$]
We use the same notation as in the proof of Lemma \ref{lem:SLWS2L}. 
The computational results 
(i), (iii) and (iv) are valid also for $g=3$. In particular, the elements 
$X_i^2 \wedge X_{ij}$, $X_{ij} \wedge X_{jk}$ and 
$X_i^2 \wedge X_k^2 + X_i^2 \wedge X_{jk}$ are in $Q_1$. 
We also see that 
\begin{align*}
& [e_{ki};X_i^2 \wedge X_{ij}]=X_k^2 \wedge X_{kj}+ 
X_k^2 \wedge X_{ij} +X_{ik} \wedge X_{kj} 
+X_{ik} \wedge X_{ij} + X_i^2 \wedge X_{kj},\\
& [e_{ji};X_i^2 \wedge X_{ij}]=2X_i^2 \wedge X_j^2+X_{ij} \wedge X_j^2.
\end{align*}
\noindent
are in $Q_1$, from which 
$X_k^2 \wedge X_{ij} + X_i^2 \wedge X_{kj}$ and 
$2X_i^2 \wedge X_j^2$ are in $Q_1$. 
Then there remains only two possibilities: 
$(\wedge^2(S^2 L))_{SL(3,\Z)}=0$ or $\Z_2$ generated by 
\[X_1^2 \wedge X_{23}=X_3^2 \wedge X_{12}=X_2^2 \wedge X_{13}=
X_1^2 \wedge X_2^2=X_1^2 \wedge X_3^2 = X_2^2 \wedge X_3^2.\]
By using Theorem \ref{thm:Stein}, we see that the latter is true. 
Indeed, the element $X_2^2 \wedge X_3^2$ 
just maps to the element of order $2$ in $H_2 (Sp(6,\Z))$ by the map 
$(\wedge^2(S^2 L))_{SL(3,\Z)}=H_2(S^2 L)_{SL(3,\Z)} \to H_2 (Sp(6,\Z))$.  \end{proof}

\begin{proof}[Proof of Proposition $\ref{lem:urSp+}$ for $g=3$]
The $E^2$-term of the Lyndon-Hochschild-Serre spectral sequence 
associated with the split extension (\ref{seq:urSp+}) is given by 
\[E^2 = {\small \begin{array}{|c|c|c|c}
&&& \\
\hline 
(\wedge^2(S^2 L))_{SL(3,\Z)} \cong \Z_2 &&& \\
\hline
(S^2 L)_{SL(3,\Z)}=0 & H_1 (SL(3,\Z);S^2 L) \cong \Z_2 & 
H_2 (SL(3,\Z);S^2 L) & \\
\hline
\Z & H_1 (SL(3,\Z))=0 & H_2 (SL(3,\Z)) \cong \Z_2^2 & H_3 (SL(3,\Z))\\
\hline
\end{array}}\ .
\]
By Lemma \ref{lem:genus3_2}, 
the generator of $(\wedge^2(S^2 L))_{SL(3,\Z)}=\Z_2$ 
survives in $H_2 (\urp{6})$. Therefore 
$d_2: H_2 (SL(3,\Z);S^2 L) \to (\wedge^2(S^2 L))_{SL(3,\Z)}$ is a trivial map. 
The existence of the splitting of the extension (\ref{seq:urSp+}) 
shows that 
$d_2: H_3(SL(3,\Z)) \to H_1 (SL(3,\Z);S^2 L)$ and 
$d_3: H_3(SL(3,\Z)) \to (\wedge^2(S^2 L))_{SL(3,\Z)}$ are also trivial. 
Hence $E_{p,q}^2 = E_{p,q}^\infty$ for $p+q \le 2$. 
The $E^\infty$-term says that there exists a filtration 
\[H_2(\urp{6}) \supset F_0 \supset F_1 = E_{0,2}^\infty\]
with $H_2 (\urp{6})/F_0 \cong E_{2,0}^\infty$ and 
$F_0/F_1 \cong E_{1,1}^\infty$. Again the existence of the splitting of 
the extension (\ref{seq:urSp+}) shows that 
$H_2 (\urp{6}) \cong F_0 \oplus E_{2,0}^\infty \cong 
F_0 \oplus H_2 (SL(3,\Z))$. Finally we consider the extension
\[0 \longrightarrow (\wedge^2(S^2 L))_{SL(3,\Z)} \cong \Z_2 \longrightarrow 
F_0 \longrightarrow H_1 (SL(3,\Z);S^2 L) \cong \Z_2 \longrightarrow 0.\]
Suppose $F_0 \cong \Z_4$. Then the second map $\Z_2 \to \Z_4$ should send 
$1 \in \Z_2$ to $2 \in \Z_4$. 
This contradicts to the fact that 
the generator of $(\wedge^2(S^2 L))_{SL(3,\Z)}$ 
maps to 
the element of order $2$ in $H_2 (Sp(6,\Z))\cong \Z \oplus \Z_2$.
Therefore $F_0 \cong \Z_2 \oplus \Z_2$ and we finish the proof. 
\end{proof}
\begin{remark}
The homology of $SL(3,\Z)$ was completely determined by Soul\'e \cite{soule}. 
In particular, $H_3 (SL(3,\Z)) \cong \Z_3^2 \oplus \Z_4^2$. 
\end{remark}

We finish this section by pointing out a by-product of 
our argument (see also Remark \ref{rem:H2M3_2}). 
Consider the second homology of 
the full mapping class group $\mathcal{M}_{3,1}$ of genus $3$. 
In \cite{ks}, Korkmaz-Stipsicz showed that 
$H_2 (\mathcal{M}_3)$ is $\Z$ or $\Z \oplus \Z_2$. 
Now we can use Lemma \ref{lem:H2ILgb} and the fact that the generator of 
$(\wedge^2(S^2 L))_{SL(3,\Z)} \cong \Z_2$ 
maps to the element of order $2$ in $H_2 (Sp (6,\Z)) \cong \Z \oplus \Z_2$ 
to show that there exists an element of 
$H_2 (\mathcal{M}_{3,1})$ which comes from $H_2 (\ILgb)$ and 
maps to the element of order $2$ in $H_2 (Sp (6,\Z))$. Consequently, 
we have: 
\begin{thm}\label{thm:H2M3b}
$H_2 (\mathcal{M}_{3,1}) \cong \Z \oplus \Z_2$
\end{thm}
\noindent
By using an argument of Korkmaz-Stipsicz in \cite{ks}, 
we can derive the following.
\begin{cor}\label{cor:H2M3}
$H_2 (\mathcal{M}_3) \cong \Z \oplus \Z_2$ 
and  $H_2 (\mathcal{M}_{3,\ast}) \cong \Z \oplus \Z \oplus \Z_2$, where 
$\mathcal{M}_{g,\ast}$ denotes the mapping class group of a surface 
of genus $g$ with one puncture. 
\end{cor}

\section{The first and second homology of $\Lgb$}\label{sec:H1Lgb}
We use our results in the previous sections to determine 
$H_1 (\Lgb)$ and give a lower bound of $H_2 (\Lgb)$. 
\begin{thm}\label{thm:H1Lgb}
\begin{tabular}[t]{ll}
$(1)$ & $H_1 (\Lgb) \cong 
\left\{\begin{array}{ll}
\Z_2 \oplus \Z_2  & (g = 3), \\
\Z_2 & (g \ge 4).
\end{array}\right.$\\
$(2)$ & The map $(\sigma|_{\Lgb})_\ast: 
H_2 (\Lgb) \to H_2(\ur{2g})$ is surjective for $g \ge 3$. 
\end{tabular}
\end{thm}
\begin{proof}
We consider the five-term exact sequence 
\begin{equation}\label{seq:5TLgb}
H_2 (\Lgb) \to H_2 (\ur{2g}) \to H_1 (\Igb)_{\ur{2g}} 
\to H_1 (\Lgb) \to H_1(\ur{2g}) \to 0
\end{equation}
\noindent
associated with the group extension (\ref{seq:ur}). 
We have seen that $H_1(\ur{2g}) \cong \Z_2$. 
We now show that 
\[H_1 (\Igb)_{\ur{2g}} \cong 
\left\{\begin{array}{ll}
\Z_2 & (g = 3), \\
0 & (g \ge 4),
\end{array}\right.\]
which proves the theorem for $g \ge 4$. 

We put $H_1 (\Igb)_{\ur{2g}}=H_1 (\Igb)/Q_2$ with 
$Q_2$ generated by $\{[\sigma; x] \mid \sigma \in 
\ur{2g}, x \in H_1 (\Igb)\}$. Note that 
$Q_2$ includes $Q_0$ in 
the proof of Lemma \ref{lem:Ig_Coinv} 
since $S^2 L \subset \ur{2g}$. 

We have 
\begin{align*}
&[e_{kl}^{-1} \oplus e_{lk};
(y_i \wedge y_j \wedge y_k, \overline{y}_i \overline{y}_j\overline{y}_k)]
=(y_i \wedge y_j \wedge y_l, \overline{y}_i \overline{y}_j\overline{y}_l)\\
\intertext{for $g \ge 4$ and also have}
&[e_{ik}^{-1} \oplus e_{ki};
(y_i \wedge x_j \wedge y_j, \overline{y}_i \overline{x}_j\overline{y}_j)]
=(y_k \wedge x_j \wedge y_j, \overline{y}_k \overline{x}_j\overline{y}_j)
\end{align*}
\noindent
for $g \ge 3$. Hence $H_1 (\Igb)/Q_2 =0$ holds for $g \ge 4$. In the case where 
$g =3$, we have 
\begin{align*}
&\left[ {\scriptsize \begin{pmatrix} 
0 & 0 & 1 \\ 0 & 1 & 0 \\ 1 & 0 & 0\end{pmatrix}} \oplus 
{\scriptsize \begin{pmatrix} 
0 & 0 & 1 \\ 0 & 1 & 0 \\ 1 & 0 & 0\end{pmatrix}};
(y_1 \wedge y_2 \wedge y_3, \overline{y}_1 \overline{y}_2 \overline{y}_3)
\right]
=-2(y_1 \wedge y_2 \wedge y_3, \overline{y}_1 \overline{y}_2 \overline{y}_3),\\
&[e_{ik}^{-1} \oplus e_{ki};
(0 , \overline{y}_i \overline{y}_j)]
=(0, \overline{y}_k \overline{y}_j).
\end{align*}
\noindent
Therefore $H_1 (\mathcal{I}_{3,1})/Q_2$ is at most $\Z_2$ generated by 
$(y_1 \wedge y_2 \wedge y_3, \overline{y}_1 \overline{y}_2 \overline{y}_3)$. 
To see that $H_1 (\mathcal{I}_{3,1})/Q_2 \cong \Z_2$, which proves (1) 
and (2) for $g=3$ simultaneously, we now show that there exists a 
splitting $H_1 (\mathcal{L}_{3,1}) \to H_1 (\mathcal{I}_{3,1})/Q_2$ by 
constructing a homomorphism $H_1 (\mathcal{L}_{3,1}) \to \Z_2$ 
whose precomposition by 
$H_1 (\mathcal{I}_{3,1}) \to H_1 (\mathcal{L}_{3,1})$ is non-trivial. Indeed 
if such a homomorphism exists, $H_1 (\mathcal{I}_{3,1})/Q_2 \cong \Z_2$ 
immediately follows and the composition 
$H_1 (\mathcal{I}_{3,1})/Q_2 \to H_1 (\mathcal{L}_{3,1}) \to \Z_2 
\cong H_1 (\mathcal{I}_{3,1})/Q_2$ becomes the identity map. 

Our construction uses the extended Johnson homomorphism 
\[\rho=(\widetilde{k},\sigma):
\mathcal{M}_{3,1} \longrightarrow \frac{1}{2} \wedge^3 H \rtimes 
Sp(6,\Z)\]
first defined by Morita \cite{mo5}. 
Note that $\widetilde{k}: \mathcal{M}_{3,1} \to 
\frac{1}{2} \wedge^3 H$ is 
a crossed homomorphism which extends the original 
Johnson homomorphism $\tau: \mathcal{I}_{3,1} \to \wedge^3 H$. 
Precisely speaking, such an extension $\widetilde{k}$ is not unique but 
unique up to certain coboundaries 
(see \cite[Sections 4, 5]{mo5} 
for details). Here we use 
the formulation by Birman-Brendle-Broaddus in 
\cite[Section 2.2]{bbb} and denote their crossed homomorphism 
by $\widetilde{k}: \mathcal{M}_{3,1} \to \frac{1}{2} \wedge^3 H$ again. 

Consider the composition
\[\psi: \mathcal{L}_{3,1} \xrightarrow{\widetilde{k}|_{\Lgb}} 
\frac{1}{2} \wedge^3 H \xrightarrow{\mathrm{proj}} 
\frac{1}{2} \wedge^3 L \cong 
\frac{1}{2} \wedge^3 \Z^3 \cong \frac{1}{2} \Z 
\longrightarrow \left(\frac{1}{2} \Z \right) / (2 \Z),\]
where the second map is induced from the projection 
$H \twoheadrightarrow L$ (in other word, this map assigns 
the coefficient of $y_1 \wedge y_2 \wedge y_3$ under 
our basis of $H$). 
We claim the following: 
\begin{itemize}
\item[(i)] \ $\Im \psi \subset \Z/2\Z = \Z_2$

\item[(ii)] \ $\psi: \mathcal{L}_{3,1} \to \Z_2$ is a homomorphism

\item[(iii)] \ the composition $\mathcal{I}_{3,1} 
\to \mathcal{L}_{3,1} \xrightarrow{\psi} \Z_2$ is non-trivial
\end{itemize}
\noindent
To show (i), we recall that $\mathcal{L}_{3,1} 
= \mathcal{I}_{3,1} \mathcal{H}_{3,1}$, where 
$\mathcal{H}_{3,1}$ is the preimage of the handlebody 
mapping class group $\mathcal{H}_3$ of genus $3$ by 
the natural homomorphism 
$\mathcal{M}_{3,1} \to \mathcal{M}_3$. 
Birman-Brendle-Broaddus showed 
in \cite[Section 2.2]{bbb} that 
$\widetilde{k}(h)$ does not have the term 
$n \, y_1 \wedge y_2 \wedge y_3$ with $n \in \frac{1}{2} 
\mathbb{Z}-\{0\}$ 
for any $h \in \mathcal{H}_{3,1}$. Since 
\[\widetilde{k}(f)=\widetilde{k}(ih) = 
\widetilde{k}(i)+{}^{\sigma (i)}\widetilde{k}(h)=
\widetilde{k}(i)+\widetilde{k}(h)\]
for any element $f=ih \in \mathcal{L}_{3,1}$ with 
$i \in \mathcal{I}_{3,1}$ and $h \in \mathcal{L}_{3,1}$, 
and $\widetilde{k}(i)=\tau(i) \in \wedge^3 H$, 
we see that $\psi (f)=\psi (i)+\psi (h) = \psi (i) 
\in \Z/2\Z$, which proves (i). 
Next, (ii) follows from the facts that $\mathcal{L}_{3,1}$ acts on 
$H$ with keeping $L$ and acts on $L$ through 
$ul \circ \sigma |_{\Lg}:\mathcal{L}_{3,1} \to GL(3,\Z)$ and that 
$GL(3,\Z)$ acts on $\wedge^3 L \cong \Z$ through 
$\det :GL(3,\Z) \to \{1,-1\}$. Finally, 
(iii) clearly follows from the 
construction and we finish the proof.
\end{proof}
\begin{remark}\label{rem:H2M3_2}
The above computation of $H_1 (\mathcal{I}_{3,1})_{\ur{6}}$ and 
the equality
\[\left[{\scriptsize 
\begin{pmatrix} 
1 & 0 & 0 & 0 & 0 & 0\\
0 & 1 & 0 & 0 & 0 & 0\\
0 & 0 & 1 & 0 & 0 & 0\\
0 & 0 & 0 & 1 & 0 & 0\\
0 & 0 & 0 & 0 & 1 & 0\\
0 & 0 & 1 & 0 & 0 & 1\\
\end{pmatrix}}; (y_1 \wedge y_2 \wedge x_3, 
\overline{y}_1 \overline{y}_2 \overline{x}_3)
\right]=
(y_1 \wedge y_2 \wedge y_3, 
\overline{y}_1 \overline{y}_2 \overline{y}_3
+\overline{y}_1 \overline{y}_2)\]
show that $H_1 (\mathcal{I}_{3,1})_{Sp(6,\Z)}=0$ 
(see also Putman \cite[Lemma 6.4]{pu}). 
Then by the five term exact sequence associated with (\ref{seq:mg}), 
the map $H_2 (\mathcal{M}_{3,1}) \to H_2 (Sp(6,\Z))$ is onto. 
Therefore, by using the results of Korkmaz-Stipsicz and Stein 
mentioned in Section \ref{sec:H_ur}, 
we can obtain another proof of 
$H_2 (\mathcal{M}_{3,1}) \cong \Z \oplus \Z_2$. 
\end{remark}
\begin{remark}
We have seen in Section \ref{sec:H_ur} that 
$\displaystyle\lim_{g \to \infty} H_2(urSp(2g)) \cong \displaystyle\lim_{g \to \infty} H_2(GL(g,\Z))$. 
The stable homology 
$\displaystyle\lim_{g \to \infty} H_2(GL(g,\Z)) \cong \Z_2$ also relates 
to the second homology of the automorphism group of a free group as 
shown by Gersten \cite{gersten}. 
\end{remark}

\section{Results for Lagrangian mapping class groups of 
closed surfaces}\label{sec:closed}

We now consider the Lagrangian mapping class groups $\Lg$ and $\ILg$ for 
closed surfaces. The relationship of $\Lgb$ and $\Lg$ is given by 
the exact sequence
\[1 \longrightarrow \pi_1 (T_1 \Sg) \longrightarrow 
\Lgb \longrightarrow \Lg \longrightarrow 1,\]
where $T_1 \Sg$ is the unit tangent bundle of $\Sg$ (see \cite{bi}), and 
the relationship of $\ILgb$ and $\ILg$ 
is obtained by replacing $\Lgb$ and $\Lg$ 
with $\ILgb$ and $\ILg$. As a subgroup of $\Lgb$ and $\ILgb$, 
the group 
$\pi_1 (T_1 \Sg) \subset \Igb$ is generated by the Dehn twist along the 
boundary curve of $\Sgb$ and spin-maps 
(see Birman's book \cite[Theorem 4.3]{bi} and 
Johnson \cite[Section 3]{jo1} for example). 
\begin{thm}\label{thm:H1ILg}
$H_1 (\ILg) \cong \left\{\begin{array}{ll}
\wedge^3 L^\ast \oplus \wedge^2 (L^\ast \otimes \Z_2) \oplus S^2 L & 
(g = 3), \\
\wedge^3 L^\ast \oplus S^2 L & 
(g \ge 4).\end{array}\right.$
\end{thm}
\begin{proof}
We have an exact sequence
\[H_1 (\pi_1 (T_1 \Sg)) \longrightarrow 
H_1 (\ILgb) \longrightarrow H_1 (\ILg) \longrightarrow 0.\]
In \cite[Section 3.4]{le1}, Levine showed that $\pi_1 (T_1 \Sg)$ 
projects trivially on $\wedge^3 L^\ast$ and onto on $L^\ast$ 
with respect to 
the abelianization 
\[H_1 (\ILgb) \cong \left\{\begin{array}{ll}
\wedge^3 L^\ast \oplus L^\ast \oplus 
\wedge^2 (L^\ast \otimes \Z_2) \oplus S^2 L & 
(g = 3), \\
\wedge^3 L^\ast \oplus L^\ast \oplus S^2 L & 
(g \ge 4).\end{array}\right. \]
Since $\pi_1 (T_1 \Sg)$ is included in $\Igb$, 
it projects trivially on $S^2 L$. 
Hence the theorem for $g \ge 4$ holds. In the case where $g=3$, 
we can directly check that all of generators of $\pi_1 (T_1 \Sigma_3)$ 
are sent to $0 \in \wedge^2 (L^\ast \otimes \Z_2)$, which completes the 
proof for $g=3$. 
\end{proof}

\begin{thm}\label{thm:H1Lg}
\begin{tabular}[t]{ll}
$(1)$ & $H_1 (\Lg) \cong H_1 (\Lgb) \cong 
\left\{\begin{array}{ll}
\Z_2 \oplus \Z_2  & (g = 3), \\
\Z_2 & (g \ge 4).
\end{array}\right.$\\
$(2)$ & The map $(\sigma|_{\Lg})_\ast: 
H_2 (\Lg) \to H_2(\ur{2g})$ is surjective for $g \ge 3$. 
\end{tabular}
\end{thm}
\begin{proof}
Since $\sigma |_{\Lgb}: \Lgb \to \ur{2g}$ 
factors through $\Lg$, (1) for $g \ge 4$ 
immediately holds. (1) for $g=3$ also holds by explicit computations of 
the extended Johnson homomorphism for generators of $\pi_1 (T_1 \Sigma_3)$. 
The proof of (2) is the same as that of Theorem \ref{thm:H1Lgb}. 
\end{proof}

\section{Remarks on higher (co)homology of $\Lg$ and $\ILg$}\label{sec:remark}
\subsection{Relationship to the homology of the pure braid group}\label{subsec:pure}

In \cite{le1}, Levine studied various embeddings of 
the pure braid group $P_n$ of 
$n$ strands into $\Mgb$ and $\Mg$, where 
$n=g, 2g$ etc. We now use one of them defined as follows. 
Let $D_g$ be a disk with $g$ holes. We take an 
embedding $\iota : D_g \hookrightarrow \Sgb$ as in Figure \ref{fig:embedding}, 
where we consider the surface $\Sgb$ to be a disk with $g$ handles attached and 
the belt circles of the handles correspond to the loops $x_1,x_2,\ldots,x_g$  
in Figure \ref{fig:surface} after filling the boundary $\partial \Sgb$ by a disk. 
The mapping class group of 
$D_g$, where the self-diffeomorphisms of $D_g$ are supposed to 
fix the boundary pointwise, is known to be isomorphic to 
the framed pure braid group 
of $g$ strands. Here the framing counts how many times 
one gives Dehn twists along each of the loops parallel to 
the inner boundary. 
For any choice of framings, we have a homomorphism 
from the pure braid group $P_g$ of $g$ strands to $\Mgb$ by extending each 
mapping class by identity on the outside of $\iota (D_g)$. 
We can easily check that 
the image of this map is contained in $\ILgb$. 
Therefore we obtain a homomorphism $\Phi: P_g \to \ILgb$. 
Similarly, we have a homomorphism from $P_g$ to $\ILg$ 
also denoted by $\Phi: P_g \to \ILg$. 

\begin{figure}[htbp]
\begin{center}
\includegraphics[width=0.6\textwidth]{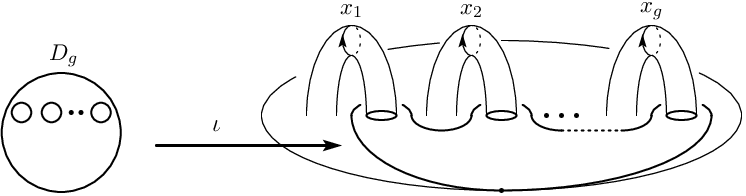}
\end{center}
\caption{The embedding $\iota :D_g \hookrightarrow \Sgb$}
\label{fig:embedding}
\end{figure}

\begin{thm}\label{thm:pure}
The induced map $\Phi_\ast :H_\ast (P_{g}) \to H_\ast (\ILg)$ is injective. 
\end{thm}
\begin{proof}
Consider the induced map $H^\ast (S^2 L) \to H^\ast (P_g)$ of 
the composition $P_g \xrightarrow{\Phi} \ILg \to S^2 L$ on cohomology. 
Here the ring structure of $H^\ast (P_g)$ was 
completely determined by Arnol'd in \cite{arnold}, 
and in particular, it was shown that 
$H^\ast (P_g)$ is a finitely generated free abelian group and 
is generated by degree $1$ elements as a ring. 
The former shows that $H_\ast (P_g)$ is also finitely generated free abelian 
and the latter shows that $H^\ast (S^2 L) \to H^\ast (P_g)$ is onto since 
it is clear from a presentation of $P_g$ (see \cite{bi} for example) 
that $H^1 (S^2 L) \to H^1 (P_g)$ is onto. By passing to homology, 
we see that $H_\ast (P_g) \to H_\ast (S^2 L)$ is injective. 
The theorem follows from this. 
\end{proof}

\subsection{Vanishing of odd Miller-Morita-Mumford classes on $\Lg$}
\label{subsec:oddMMM}

Finally, we discuss the rational cohomology of higher degrees 
of $\Lg$ with relationships to characteristic classes of oriented 
$\Sg$-bundles called {\it Miller-Morita-Mumford classes}. 

Here we recall the definition of Miller-Morita-Mumford classes 
following Morita \cite{mo1}. 
Let $\pi: E \to B$ be an oriented $\Sg$-bundle over a closed oriented 
manifold $B$. Since $\Sg$ is $2$-dimensional, the relative tangent bundle 
$\Ker \pi_\ast$ is a vector bundle over $E$ of rank $2$. In particular, 
we can take its Euler class $e \in H^2 (E)$. Then 
$i$-th Miller-Morita-Mumford class $e_i$ is defined by 
\[e_i := \pi_! (e^{i+1}) \in H^{2i} (B),\]
where $\pi_! : H^\ast (E) \to H^{\ast -2} (B)$ is the Gysin map. 
This construction is natural with respect to bundle maps, 
so that we can regard $e_i$ as 
a cohomology class in the classifying space. 
Namely $e_i \in H^{2i}(B\mathrm{Diff}_+ \Sg)$, 
where $B\mathrm{Diff}_+ \Sg$ is the classifying space of 
the topological group $\mathrm{Diff}_+ \Sg$ of 
orientation preserving self-diffeomorphisms of $\Sg$ 
with $C^\infty$-topology. By a theorem of Earle-Eells \cite{ee}, 
we have $B\mathrm{Diff}_+ \Sg = K(\Mg,1)$. Therefore 
\[e_i \in H^{2i}(B\mathrm{Diff}_+ \Sg) = H^{2i} (K(\Mg,1)) 
=H^{2i} (\Mg).\]

Now we ask whether $e_{i} \in H^{2i} (\Mg;\Q)$, 
regarded as a rational cohomology class, survives 
in $H^{2i} (\Ig;\Q)$ by the pull-back of $\Ig \hookrightarrow \Mg$. 
A partial answer to this question is given as follows (see Morita \cite{mo1}). 
It is known that every {\it odd} class $e_{2i-1} \in H^{4i-2} (\Mg;\Q)$ 
can be obtained as the pull-back of some class in 
$H^{4i-2} (\Symp{Z};\Q)$, which implies that  
all the odd classes $e_{2i-1}$ vanish in $H^{4i-2} (\Ig;\Q)$. However, 
this argument says nothing about {\it even} classes $e_{2i}$ and 
it has been a long standing problem to determine whether 
even classes $e_{2i}$ vanish or not in $H^{4i} (\Ig ;\Q)$. 

The author's motivation for the study in this paper 
is to attack this problem by 
considering groups locating between $\Mg$ and $\Ig$ and investigating 
the behavior of $e_i$ on them. 
As examples of such a kind of groups, finite index subgroups including 
level $L$ mapping class groups defined as the kernel of the 
composition 
\[\Mg \longrightarrow \Symp{Z} \longrightarrow 
Sp(2g,\Z/L\Z)\]
are often studied. However, 
we cannot solve the above problem by using them 
since for any finite index subgroup $G$ of $\Mg$
there exists a transfer map
\[\mathrm{tr}:H^\ast (G;\Q) \longrightarrow H^\ast (\Mg;\Q)\]
such that $\mathrm{tr} \circ i^\ast: H^\ast (\Mg;\Q) \to H^\ast (\Mg;\Q)$ 
is the multiplication by a positive integer $[\Mg:G]$, 
where $i:G \hookrightarrow \Mg$ denotes the inclusion. 
In particular, we see that the pull-back map on 
the rational cohomology is always injective for any finite index subgroup. 
Therefore we shall need infinite index subgroups and 
we focus on $\Lg$ and $\ILg$ in this paper. 
At present, we cannot give the final answer even for $\Lg$, but 
we now present an observation for odd classes, by which we finish 
this paper. 

\begin{lem}\label{lem:ur_GL}
If $g$ is sufficiently larger than $q$, 
we have
\[H^q (\ur{2g};\Q) \cong H^q (GL(g,\Z);\Q).\]
\end{lem}
\begin{proof}
The $E_2$-term of the Lyndon-Hochschild-Serre spectral sequence for 
the group extension (\ref{seq:S2ur}) is given by 
\[E_2^{p,q} = H^p (GL(g,\Z);H^q (S^2 L;\Q)).\]
Our claim immediately follows once 
we show that $H^p (GL(g,\Z);H^q (S^2 L;\Q))=0$ if $q \ge 1$. 
Since 
$H^q(S^2 L;\Q) \cong \wedge^q (S^2 (L^\ast \otimes \Q))$ and 
it is easy to show that the invariant part 
$\wedge^q (S^2 (L^\ast \otimes \Q))^{GL(g,\Z)}$ is trivial, 
we can use Borel's vanishing theorem \cite{borel} to show that 
\[H^p (GL(g,\Z);H^q (S^2 L;\Q))=0\]
for any $q \ge 1$.
\end{proof}

\begin{thm}
For every $i$, the $(2i-1)$-st 
Miller-Morita-Mumford class $e_{2i-1}$ vanishes in 
$H^{4i-2} (\Lg;\Q)$ if $g$ is 
sufficiently larger than $i$.
\end{thm}
\begin{proof} 
It is known that the group cohomology $H^\ast (G)$ 
of a discrete group $G$ can be rewritten as 
$H^\ast (BG^\delta)$, where $BG$ denotes the classifying space of $G$. 
When $G$ is a Lie group, we write $G^{C^\infty}$ for $G$ with 
$C^\infty$ topology and $G^\delta$ for $G$ with discrete topology. 

Consider the following commutative diagram:

\[\SelectTips{cm}{}\xymatrix{
H^\ast (BSp(2g,\mathbb{R})^{C^\infty};\Q) \ar[rr] 
\ar[d]^{B (\mathrm{id})^\ast} 
\ar@{}[ddrr]|\circlearrowleft & & 
H^\ast (BGL(g,\mathbb{R})^{C^\infty};\Q) 
\ar[d]^{B (\mathrm{id})^\ast} \\
H^\ast (BSp(2g,\mathbb{R})^{\delta};\Q) \ar[d] & & 
H^\ast (BGL(g,\mathbb{R})^{\delta};\Q) \ar[d] \\
H^\ast (BSp (2g,\Z);\Q) \ar[d]^{B\sigma^\ast} \ar[r] 
\ar@{}[dr]|\circlearrowleft & 
H^\ast (BurSp (2g);\Q) \ar[d]^{B(\sigma|_{\Lg})^\ast} \ar[r]& 
H^\ast (BGL (g,\Z);\Q) \\
H^\ast (B\Mg);\Q) \ar[r] & H^\ast (B \Lg;\Q) & }\]
where $\mathrm{id}$ denotes the identity map, which always gives 
a continuous map 
from $G$ with discrete topology to that with $C^\infty$ topology for a 
Lie group $G$, and all the arrows not labeled 
are pull-backs by 
the induced maps of inclusions on classifying spaces. 
We now assume that $g$ is sufficiently large. 
Since $Sp(2g,\mathbb{R})^{C^\infty}$ is homotopy equivalent to 
the unitary group $U(g)^{C^\infty}$, we have 
$H^\ast (BSp(2g,\mathbb{R})^{C^\infty};\Q) \cong 
H^\ast (BU(g)^{C^\infty};\Q)$ and 
the latter is known to be isomorphic to 
the polynomial algebra $\Q [c_1, c_2, \ldots]$ generated 
by the Chern classes $c_1, c_2, \ldots$ independently 
in the stable range. This polynomial algebra $\Q [c_1, c_2, \ldots]$ is 
mapped onto $\Q [c_1, c_3, c_5, \ldots]$ 
in $H^\ast (BSp(2g,\mathbb{R})^{\delta};\Q)$, 
and onto $\Q [e_1, e_3, e_5, \ldots]$ in $H^\ast (B \Mg;\Q)$. 
We refer to \cite{mo1} again for these arguments. On the other hand, 
it was shown by Milnor \cite[Appendix]{milnor} that 
\[B (\mathrm{id})^\ast: H^\ast (BGL(g,\mathbb{R})^{C^\infty};\Q) 
\longrightarrow 
H^\ast (BGL(g,\mathbb{R})^\delta;\Q)\] 
is trivial for $\ast \ge 1$. 
By combining this fact with Lemma \ref{lem:ur_GL}, 
the theorem follows. 
\end{proof}

\begin{remark}
Recently, Giansiracusa and Tillmann \cite{gt} 
have proved a closely related result that odd 
Miller-Morita-Mumford classes vanish in the {\it integral} cohomology of 
the handlebody subgroup $\mathcal{H}_g$ for $g \ge 2$. 
In fact, they showed that 
odd Miller-Morita-Mumford classes are in the kernel of the pull-back map 
on the integral cohomology by 
$B \mathrm{Diff}_+ M \to B\mathrm{Diff}_+ \Sigma_g$ 
where $M$ is any compact oriented 3-manifold $M$ with $\partial M = \Sigma_g$. 
\end{remark}

\section{Acknowledgement}\label{sec:acknowledge}
The author would like to express his gratitude 
to Professor Sadayoshi Kojima and Professor Shigeyuki Morita 
for their encouragement and to Yusuke Kuno and Andrew Putman 
for useful discussions and comments. 
The author also would like to thank the referee for 
his-or-her helpful comments to improve the paper. 
This research was supported by 
JSPS Research Fellowships for Young Scientists and 
KAKENHI (No.~08J02356 and No.~21740044).

\end{document}